\documentclass[11pt,leqno]{article}
\usepackage{xcolor}
\definecolor{labelkey}{rgb}{0,0.08,0.45}
\definecolor{refkey}{rgb}{0,0.6,0.0}
\definecolor{Brown}{rgb}{0.45,0.0,0.05}
\usepackage{mathpazo}
\usepackage{exscale,relsize}
\usepackage{amsmath}
\usepackage{amsfonts}
\usepackage{amssymb}
\usepackage{calc}
\usepackage{theorem}
\usepackage{pifont}      %needed by dingautolist
\usepackage{graphicx}
\usepackage{soul,color}
\newcommand{\HH}{\ensuremath{X}}
\newcommand{\gr}{\ensuremath{\operatorname{gra}}}

\newcommand{\scal}[2]{\langle{{#1},{#2}}\rangle}

\newcommand{\NN}{\ensuremath{\mathbb N}}
\newcommand{\nnn}{\ensuremath{{n \in \NN}}}
\newcommand{\thalb}{\ensuremath{\tfrac{1}{2}}}

\newcommand{\menge}[2]{\big\{{#1} \mid {#2}\big\}}

\newcommand{\To}{\ensuremath{\rightrightarrows}}

\newcommand{\zer}{\ensuremath{\operatorname{zer}}}

\newcommand{\Fix}{\ensuremath{\operatorname{Fix}}}

\newcommand{\Id}{\ensuremath{\operatorname{Id}}}

\newcommand{\weakly}{\ensuremath{\,\rightharpoonup}\,}

\renewcommand{\phi}{\ensuremath{\varphi}}

\newcommand{\bX}{\ensuremath{\mathbf{X}}}

\newcommand{\bD}{\ensuremath{\mathbf{D}}}
\newcommand{\bF}{\ensuremath{\mathbf{F}}}

\newcommand{\by}{\ensuremath{\mathbf{y}}}
\newcommand{\bx}{\ensuremath{\mathbf{x}}}
\newcommand{\bz}{\ensuremath{\mathbf{z}}}
\newcommand{\bC}{\ensuremath{\mathbf{C}}}

\newtheorem{theorem}{Theorem}[section]

\newtheorem{fact}[theorem]{Fact}
\newtheorem{corollary}[theorem]{Corollary}

\theoremstyle{plain}{\theorembodyfont{\rmfamily}
}
\theoremstyle{plain}{\theorembodyfont{\rmfamily}
}
\theoremstyle{plain}{\theorembodyfont{\rmfamily}
}
\theoremstyle{plain}{\theorembodyfont{\rmfamily}
\newtheorem{example}[theorem]{Example}}
\theoremstyle{plain}{\theorembodyfont{\rmfamily}
\newtheorem{remark}[theorem]{Remark}}
\theoremstyle{plain}{\theorembodyfont{\rmfamily}
}

\allowdisplaybreaks 

\begin{document}

%\sffamily

\title{\textsc{New Demiclosedness Principles
for\\ (firmly) nonexpansive operators}}

\author{
Heinz H.\ Bauschke\thanks{Mathematics \& Statistics, 
University of British Columbia, 
Kelowna, B.C.\ V1V 1V7, Canada. E-mail:
\texttt{heinz.bauschke@ubc.ca}.}
%~~and~~
%Patrick L.\ Combettes\thanks{Info...
%France. E-mail:
%\texttt{plc@math.jussieu.fr}.}
}
 \vskip 3mm

\date{March 4, 2011\\~~\\~~\\
\emph{Dedicated to Jonathan Borwein on the occasion of his 60th
Birthday}} % 
\maketitle

\begin{abstract} \noindent
The demiclosedness principle is one of the key tools
in nonlinear analysis and fixed point theory. 
In this note, this principle is extended and made more flexible 
by two mutually orthogonal affine subspaces. Versions
for finitely many (firmly) nonexpansive operators are presented.
As an application, a simple proof of the weak
convergence of the Douglas-Rachford splitting algorithm is provided.

\end{abstract}



\section{Introduction}

\noindent
Throughout this paper, we assume that
\begin{equation}
\text{
$X$ is a real Hilbert space with 
inner product $\scal{\cdot}{\cdot}$ and induced norm $\|\cdot\|$.}
\end{equation}
We shall assume basic notation and results from
Fixed Point Theory and from Monotone Operator Theory;
see, e.g., \cite{BC,BorVan,BurIus,GoeKir,GoeRei,RockWets,Si,Si2,Zalinescu}. 
The \emph{graph} of a maximally monotone operator $A\colon\HH\To\HH$ is
denoted by $\gr A$, its \emph{resolvent} $(A+\Id)^{-1}$ by $J_A$,
its set of zeros by $\zer A = A^{-1}(0)$, 
and we set $R_A = 2J_A-\Id$, where $\Id$ is the identity operator. 
Weak convergence is indicated by $\weakly$. 

Let $T\colon X\to X$.
Recall that $T$ is \emph{firmly nonexpansive} if
\begin{equation}
(\forall x\in X)(\forall y\in X)\quad
\|Tx-Ty\|^2 + \|(\Id-T)x-(\Id-T)y\|^2 \leq \|x-y\|^2.
\end{equation}
It is well know that $T$ is firmly nonexpansive if and only if
$R = 2T-\Id$ is \emph{nonexpansive}, i.e.,
\begin{equation}
(\forall x\in X)(\forall y\in X)\quad
\|Rx-Ry\|\leq \|x-y\|. 
\end{equation}
Clearly, every firmly nonexpansive operator is nonexpansive. 
Building on work by Minty \cite{Minty}, Eckstein and Bertsekas
\cite{EckBer}
clearly linked firmly nonexpansive mappings to maximally monotone operators---the
key result is the following:
$T$ is firmly nonexpansive if and only if $T=J_A$ for
some maximally monotone operator $A$ (namely, $T^{-1}-\Id$). 
Thus, finding a zero of $A$ is equivalent to finding a fixed point of
$J_A$. Furthermore, the graph of any maximally monotone operator
is beautifully described by the associated \emph{Minty parametrization}:
\begin{equation}
\label{e:Mintypar}
\gr A = \menge{(J_Ax,x-J_Ax)}{x\in X}.
\end{equation}
The most prominent example of firmly nonexpansive mappings are projectors, 
i.e., resolvents of normal cone operators associated with nonempty closed
convex subsets of $X$. Despite being (firmly) nonexpansive and hence
Lipschitz continuous, even projectors do not interact well with the weak
topology as was first observed by Zarantonello \cite{Zara}: 

\begin{example}
\label{ex:Zara}
Suppose that $X=\ell_2(\NN)$,
set $C = \menge{x\in X}{\|x\|\leq 1}$, 
and denote the sequence of standard unit vectors in $X$ by 
$(e_n)_{n\in\NN}$. 
Set $(\forall \nnn)$ $z_n = e_0+e_n$.
Then
\begin{equation}
z_n \weakly e_0
\quad\text{yet}\quad
P_Cz_n \weakly \tfrac{1}{\sqrt{2}}e_0\neq e_0 = P_Ce_0.
\end{equation}
\end{example}

The following classical demiclosedness principle dates back to the
1960s and work by Browder \cite{Browder68}. It comes somewhat as a surprise in view
of the previous example. 

\begin{fact}[Demiclosedness Principle]
Let $S$ be a nonempty closed convex subset of $\HH$,
let $T\colon S\to\HH$ be nonexpansive,
let $(z_n)_\nnn$ be a sequence in $S$ converging weakly to $z$,
and suppose that $z_n-Tz_n\to x$.
Then $z-Tz=x$.
\end{fact}

\begin{remark}
One might inquire whether or not the following even 
less restrictive demiclosedness principle holds:
\begin{equation}
\left. \begin{matrix}z_n \weakly {z}\\
z_n-Tz_n\weakly x
\end{matrix}
\right\}
\;\;\stackrel{?}{\Rightarrow}\;\;
z-Tz=x. 
\end{equation}
However, this is generalization is false in general:
indeed, suppose that $X$, $C$, and $(x_n)_\nnn$
are as in Example~\ref{ex:Zara},
and set 
$T=\Id-P_C$, which is (even firmly) nonexpansive.
Then $x_n\weakly e_0$ and $x_n-Tx_n=P_Cx_n\weakly
\tfrac{1}{\sqrt{2}}e_0$ yet
$e_0-Te_0 = P_Ce_0 = e_0 \neq \tfrac{1}{\sqrt{2}}e_0$.
\end{remark}

The aim of this note is 
to provide new versions of the demiclosedness principle
and illustrate their usefulness. 
The remainder of this paper is organized as follows.
Section~\ref{s:2} presents new demiclosedness principles
for one (firmly) nonexpansive operator.
Multi-operator versions are provided in Section~\ref{s:3}.
The weak convergence of the Douglas-Rachford algorithm
is rederived with a very transparent proof in Section~\ref{s:4}.

\section{Demiclosedness Principles}

\label{s:2}

\begin{fact}[Br\'ezis]
\label{f:Brezis}
{(\rm See \cite[Proposition~2.5 on page~27]{Brezis},
\cite[Lemma~4]{Svai10}, or \cite[Corollary~20.49]{BC}.)}
Let $A\colon X\To X$ be maximally monotone,
let $(x,u)\in X\times X$, and
let $(x_n,u_n)_\nnn$ be a sequence in $X\times X$
such that 
$(x_n,u_n)\weakly (x,u)$
and $\varlimsup \scal{x_n}{u_n} \leq \scal{x}{u}$.
Then $\scal{x_n}{u_n}\to\scal{x}{u}$ and 
$(x,u)\in \gr A$. 
\end{fact}

\begin{theorem}
\label{t:book}
{(\rm See also \cite[Proposition~20.50]{BC}.)}
Let $A\colon X\To X$ be maximally monotone,
let $(x,u)\in X\times X$, and
let $C$ and $D$ be closed affine subspaces of $X$
such that $D-D=(C-C)^\bot$.
Furthermore,
let $(x_n,u_n)_\nnn$ be a sequence in $\gr A$
such that 
\begin{equation}
(x_n,u_n)\weakly (x,u)
\quad\text{and}\quad
(x_n,u_n)-P_{C\times D}(x_n,u_n)\to (0,0).
\end{equation}
Then $(x,u)\in (C\times D)\cap \gr A$
and $\scal{x_n}{u_n}\to \scal{x}{u}$. 
\end{theorem}
\begin{proof}
Set $V=C-C$, which is a closed linear subspace. 
Since $x_n-P_Cx_n\to 0$, we have $P_Cx_n\weakly x$ and
thus $x\in C$. Likewise, $u\in D$ and hence
\begin{equation}
C = x+V
\quad\text{and}\quad
D = u + V^\bot.
\end{equation}
It follows that
\begin{equation}
P_C\colon z \mapsto P_Vz + P_{V^\bot}x
\quad\text{and}\quad
P_D\colon z \mapsto P_{V^\bot}z + P_Vu.
\end{equation}
Therefore, since $P_V$ and $P_{V^\bot}$ are weakly continuous, 
\begin{subequations}
\begin{align}
\scal{x_n}{u_n} 
&= \scal{P_Vx_n + P_{V^\bot}x_n}{P_Vu_n+P_{V^\bot}u_n}\\
&= \scal{P_Vx_n}{P_Vu_n}
+\scal{P_{V^\bot}x_n}{P_{V^\bot}u_n}\\
&= \scal{P_Vx_n}{u_n-P_{V^\bot}u_n}
+\scal{x_n-P_{V}x_n}{P_{V^\bot}u_n}\\
&= \scal{P_Vx_n}{u_n-(P_Du_n-P_{V}u)}\\
&\qquad +\scal{x_n-(P_{C}x_n-P_{V^\bot}x)}{P_{V^\bot}u_n}\\
&= \scal{P_Vx_n}{u_n-P_Du_n} + \scal{P_Vx_n}{P_Vu}\\
&\qquad + \scal{x_n-P_Cx_n}{P_{V^\bot}u_n} 
+ \scal{P_{V^{\bot}}x}{P_{V^\bot}{u_n}}\\
&\to \scal{P_Vx}{P_Vu} + \scal{P_{V^{\bot}}x}{P_{V^\bot}{u}}\\
&= \scal{x}{u}.
\end{align}
\end{subequations}
The result now follows from Fact~\ref{f:Brezis}. 
\end{proof}

\begin{remark}
Theorem~\ref{t:book}
generalizes \cite[Theorem~2]{Bau09},
which corresponds to the case $C$ is a closed
linear subspace and $D=C^\bot$ and which
was obtained by a different proof technique.
\end{remark}

\begin{corollary}[firm nonexpansiveness principle]
\label{c:fnp}
Let $F\colon\HH\to\HH$ be firmly nonexpansive,
let $(z_n)_\nnn$ be a sequence in $\HH$ such that 
$(z_n)_\nnn$ converges weakly to $z\in X$, 
suppose that $Fz_n\weakly x\in X$,
and that $C$ and $D$ are closed affine subspaces of $X$
such that $D-D=(C-C)^\bot$, 
$Fz_n-P_CFz_n\to 0$,
$(z_n-Fz_n)-P_D(z_n-Fz_n)\to 0$.
Then 
$x\in C$, $z\in x+ D$, and $x = Fz$.
\end{corollary}
\begin{proof}
Set $A = F^{-1}-\Id$ so that $J_A=F$. 
By \eqref{e:Mintypar}, $A$ is maximally monotone and 
\begin{equation}
(x_n,u_n)_\nnn := (Fz_n,z_n-Fz_n)_\nnn
\end{equation}
is a sequence in $\gr A$ that converges weakly
to $(x,z-x)$. 
Thus, by Theorem~\ref{t:book},
$x\in C$, $z-x\in D$,
and $z-x\in Ax$. 
Therefore, $z\in x+Ax$, i.e., 
$x=J_Az = Fz$. 
\end{proof}

\begin{corollary}[nonexpansiveness principle]
\label{c:np}
Let $T\colon X\to X$ be nonexpansive,
let $(z_n)_\nnn$ be a sequence in $X$ such that $z_n\weakly z$,
suppose that $Tz_n\weakly y$,
and that 
$C$ and $D$ are closed affine subspaces of $X$
such that $D-D=(C-C)^\bot$, 
$z_n + Tz_n - P_Cz_n -P_CTz_n \to 0$,
and 
$z_n-Tz_n - P_Dz_n - P_D(-Tz_n)\to 0$. 
Then 
$\thalb z + \thalb y\in C$,
$\thalb z - \thalb y \in D$,
and $y=Tz$. 
\end{corollary}
\begin{proof}
Set $F = \tfrac{1}{2}\Id + \tfrac{1}{2}T$,
which is firmly nonexpansive. 
Then $Fz_n \weakly \tfrac{1}{2}z + \frac{1}{2}y =: x$.
Since $P_C$ is affine, we get
\begin{subequations}
\begin{align}
&z_n + Tz_n - P_Cz_n -P_CTz_n \to 0\\
& \Leftrightarrow 
 z_n + Tz_n - 2\big(\tfrac{1}{2}P_Cz_n + \tfrac{1}{2}P_CTz_n\big) \to 0\\
& \Leftrightarrow 
 z_n + Tz_n - 2P_C\big(\tfrac{1}{2}z_n + \tfrac{1}{2}Tz_n\big) \to 0\\
& \Leftrightarrow  2Fz_n - 2P_CFz_n \to 0\\
&\Leftrightarrow Fz_n-P_CFz_n \to 0. 
\end{align}
\end{subequations}
Likewise, since
$z_n-Fz_n = z_n -\tfrac{1}{2}z_n-\tfrac{1}{2}Tz_n
= \tfrac{1}{2}z_n - \tfrac{1}{2}Tz_n$, we have 
\begin{subequations}
\begin{align}
&z_n-Tz_n - P_Dz_n - P_D(-Tz_n)\to 0\\
&\Leftrightarrow 
z_n-Tz_n - 2\big(\tfrac{1}{2}P_Dz_n + \tfrac{1}{2}P_D(-Tz_n)\big)\to 0\\
&\Leftrightarrow 
2(z_n-Fz_n) - 2P_D\big(\tfrac{1}{2}z_n + \tfrac{1}{2}(-Tz_n)\big)\to 0\\
&\Leftrightarrow z_n-Fz_n -P_D(z_n-Fz_n)\to 0. 
\end{align}
\end{subequations}
Thus, by Corollary~\ref{c:fnp}, 
$x\in C$, 
$z\in x+D$, and
$x=Fz$, i.e.,
$\thalb z + \thalb y\in C$,
$z\in \thalb z+ \thalb y + D$,
and 
$\thalb z + \thalb y = Fz = \thalb z + \thalb Tz$, i.e., 
$\thalb z + \thalb y\in C$,
$\thalb z - \thalb y \in D$,
and $y=Tz$. 
\end{proof}

\begin{corollary}[classical demiclosedness principle]
Let $S$ be a nonempty closed convex subset of $\HH$,
let $T\colon S\to\HH$ be nonexpansive,
let $(z_n)_\nnn$ be a sequence in $S$ converging weakly to $z$,
and suppose that $z_n-Tz_n\to x$.
Then $z-Tz=x$.
\end{corollary}
\begin{proof}
We may and do assume that $S=\HH$ (otherwise,
consider $T\circ P_S$ instead of $T$). 
Set $y=z-x$ and note that $Tz_n\weakly y$.
Now set $C=\HH$ and $D = \{x/2\}$. 
Then $D-D = \{0\} = X^\bot = (X-X)^\bot = (D-D)^\bot$,
$z_n+Tz_n-P_Cz_n-P_CTz_n\equiv 0$
and 
$z_n-Tz_n-P_Dz_n-P_D(-Tz_n)
= z_n-Tz_n-x/2-x/2 \to 0$. 
Corollary~\ref{c:np} implies
$y=Tz$, i.e., $z-x=Tz$. 
\end{proof}

\section{Multi-Operator Demiclosedness Principles}

\label{s:3}

Set
\begin{equation}
I = \{1,2,\ldots,m\},
\quad\text{where $m$ is an integer greater than or equal to $2$.}
\end{equation}
We shall work
in the product Hilbert space
\begin{equation}
\bX = X^I
\end{equation}
with induced inner product
$\scal{\bx}{\by} = \sum_{i\in I} \scal{x_i}{y_i}$
and $\|\bx\| = \sqrt{\sum_{i\in I} \|x_i\|^2}$,
where $\bx = (x_i)_{i\in I}$ and $\by = (y_i)_{i\in I}$
denote generic elements in $\bX$.

\begin{theorem}[Multi-Operator Demiclosedness Principle for Firmly Nonexpansive
Operators]~\\ 
\label{t:firm}
Let $(F_i)_{i\in I}$ be a family of firmly nonexpansive operators on $X$,
and let, for each $i\in I$, $(z_{i,n})_\nnn$ be a sequence in $\HH$ such that
for all $i$ and $j$ in $I$, 
\begin{align}
\label{e:kati:1} z_{i,n}\weakly
{z}_i\;\;\text{and}\;\;F_iz_{i,n}\weakly{x}\\
\label{e:kati:2} \sum_{i\in I}(z_{i,n}-F_iz_{i,n})\to
-m{x}+\sum_{i\in I}{z_i}\\
\label{e:kati:3} F_{i}z_{i,n}-F_jz_{j,n}\to 0.
\end{align}
Then $F_i{z}_i = {x}$, for every $i\in I$. 
\end{theorem}
\begin{proof}
Set 
$\bx = (x)_{i\in I}$,
$\bz = (z_i)_{i\in I}$,
$(\bz_n)= (z_{i,n})_\nnn$,
and 
$\bC = \menge{(y)_{i\in I}}{y\in X}$.
Then $\bz_n\weakly \bz$ and $\bC$ is a closed subspace of $\bX$
with $\bC^\perp = \menge{(y_i)_{i\in I}}{\sum_{i\in I}y_i=0}$.
Furthermore, we set $\bD = \bz-\bx+\bC^\perp$ so that
$(\bC-\bC)^\perp = \bC^\perp = \bD-\bD$,
and also $\bF\colon (y_i)_{i\in I}\mapsto (Fy_i)_{i\in I}$. 
Then $\bF$ is firmly nonexpansive on $\bX$,
and $\bF\bz_n\weakly \bx$. 
Now \eqref{e:kati:3} implies 
\begin{equation}
(\forall i\in I)\quad 
F_iz_{i,n} - \frac{1}{m}\sum_{j\in I} F_jz_{j,n}\to 0,
\end{equation}
which---when viewed in $\bX$---means that 
$\bF\bz_n - P_{\bC}\bF\bz_n\to 0$. 
Similarly, using \eqref{e:kati:2}, 
\begin{subequations}
\begin{align}
\bz_n - \bF\bz_n - P_{\bD}(\bz_n-\bF\bz_n) &=
\bz_n - \bF\bz_n - P_{\bz-\bx+\bC^\perp}(\bz_n-\bF\bz_n)\\
&= \bz_n-\bF\bz_n -\big( \bz-\bx +
P_{\bC^\perp}\big(\bz_n-\bF\bz_n-(\bz-\bx)\big)\big)\\
&=\big(\Id-P_{\bC^\perp}\big)(\bz_n-\bF\bz_n) -
\big(\Id-P_{\bC^\perp}\big)(\bz-\bx)\\
&= P_{\bC}(\bz_n-\bF\bz_n) - P_{\bC}(\bz-\bx)\\
&= \Big(\tfrac{1}{m}\sum_{i\in I}\big( z_{i,n} - F_iz_{i,n}-z_i+x\big)
\Big)_{j\in I}\\
&\to 0. 
\end{align}
\end{subequations}
Therefore, by Corollary~\ref{c:fnp}, $\bx = \bF\bz$. 
\end{proof}

\begin{theorem}[Multi-Operator Demiclosedness Principle for Nonexpansive
Operators]~\\
\label{t:nonexp}
Let $(T_i)_{i\in I}$ be a family of nonexpansive operators on $\HH$,
and let, for each $i\in I$, $(x_{i,n})_\nnn$ be a sequence in $\HH$ such that
for all $i$ and $j$ in $I$, 
\begin{align}
\label{e:stef:1} z_{i,n}\weakly
{z}_i\;\;\text{and}\;\;T_iz_{i,n}\weakly{y}_i\\
\label{e:stef:2} 
\sum_{i\in I}\big(z_{i,n}-T_iz_{i,n}\big)\to
\sum_{i\in I}\big({z_i}-{y}_i\big)\\
\label{e:stef:3} z_{i,n}-z_{j,n}+T_{i}z_{i,n}-T_jz_{j,n}\to 0.
\end{align}
Then $T_i{z}_i={y}_i$, 
for each $i\in I$. 
\end{theorem}
\begin{proof}
Set $(\forall i\in I)$ $F_i = \thalb \Id +\thalb T_i$. 
Then $F_i$ is firmly nonexpansive 
and $F_iz_{i,n} \weakly \thalb z_i + \thalb y_i$, 
for every $i\in I$. 
By \eqref{e:stef:3}, 
$0 \leftarrow 2F_{i}z_{i,n}-2F_{j}z_{j,n} 
= (z_{i,n}+T_iz_{i,n}) - (z_{j,n}+T_jz_{j,n})
\weakly (z_i+y_i)-(z_j+y_j)$,
for all $i$ and $j$ in $I$.
It follows that 
$x=\thalb z_i + \thalb y_i$ is \emph{independent}
of $i\in I$. 
Furthermore,
\begin{subequations}
\begin{align}
\sum_{i\in I} \big(z_{i,n}-F_iz_{i,n}\big)
&=\sum_{i\in I} \thalb\big(z_{i,n}-T_iz_{i,n}\big)\\
&\to\sum_{i\in I} \thalb\big(z_{i}-y_{i}\big)\\
&=\sum_{i\in I} \big(\thalb z_{i}-\big(x-\thalb z_i\big)\big)\\
&=-mx + \sum_{i\in I} z_i. 
\end{align}
\end{subequations}
Therefore, the conclusion follows from 
Theorem~\ref{t:firm}. 
\end{proof}

\section{Application to Douglas-Rachford splitting}

\label{s:4}

In this section, we assume that 
$A$ and $B$ are maximally monotone operators on $\HH$ such that
\begin{equation}
\zer(A+B)= (A+B)^{-1}(0)\neq\varnothing.
\end{equation}
We set 
\begin{equation}
T = \thalb\Id +\thalb R_BR_A = J_B(2J_A-\Id) + (\Id-J_A),
\end{equation}
which is the Douglas-Rachford splitting operator. 
See \cite{BC} for further information on this algorithm, and also
\cite{BorSims} for some results for operators that are not maximally
monotone. 
It is not hard to check
(this is implicit in \cite{LionsMercier} and \cite{EckBer}; 
see also \cite[Proposition~25.1(ii)]{BC}) 
that
\begin{equation}
J_A\big(\Fix T\big) = \zer(A+B).
\end{equation}
Now let $z_0\in\HH$ and 
define the sequence $(z_n)_\nnn$ by 
\begin{equation}
(\forall\nnn)\quad z_{n+1} = Tz_n.
\end{equation}
This sequence is very useful in determining
a zero of $A+B$ as the next result illustrates.

\begin{fact}[Lions-Mercier]
\emph{\cite{LionsMercier}}
\label{f:LionsMercier}
The sequence $(z_n)_\nnn$ converges weakly to some 
point $z\in X$ such that ${z}\in\Fix T$ and $J_A{z}\in\zer(A+B)$.
Moreover, the sequence $(J_Az_n)_\nnn$ is bounded, and every weak
cluster point of this sequence belongs to $\zer(A+B)$.
\end{fact}

Since $J_A$ is in general \emph{not} sequentially weakly continuous
(see Example~\ref{ex:Zara}), 
it is not obvious whether or not $J_Az_n\weakly J_A{z}$.
However, recently Svaiter provided a relatively complicated proof
that in fact weak convergence does hold.
As an application, we rederive the most useful instance of his
result with a considerably simpler and more conceptual proof. 

\begin{fact}[Svaiter]
\label{f:Svaiter}
\emph{\cite{Svai10}}
The sequence $(J_Az_n)_\nnn$ converges weakly to $J_Az$.\\
\end{fact} 
\begin{proof}
By Fact~\ref{f:LionsMercier}, 
\begin{equation}
z_n\weakly z\in\Fix T.
\end{equation}
Since $J_A$ is (firmly) nonexpansive and $(z_n)_\nnn$ is bounded,
the sequence $(J_Az_n)_\nnn$ is bounded as well.
Let $x$ be an arbitrary weak cluster point of $(J_Az_n)_\nnn$,
say 
\begin{equation}
J_Az_{k_n} \weakly x \in \zer(A+B)
\end{equation}
by Fact~\ref{f:LionsMercier}. 
Set $(\forall\nnn)$ $y_n=R_Az_n$.
%Since $R_A$ is nonexpansive, the sequence $(y_n)_\nnn$ is bounded.
%After passing to another subsequence and relabeling if necessary, 
%we assume that 
Then
\begin{equation}
y_{k_n}\weakly y = 2x-z\in X. 
\end{equation}
It is well known that the sequence of iterates of any firmly nonexpansive operator
with fixed points is asymptotically regular \cite{BruckReich}; thus, 
\begin{equation}
J_Az_n - J_By_n = z_n-Tz_n \to 0
\end{equation}
and hence 
\begin{equation}
J_By_{k_n} \weakly x.
\end{equation}
Next,
\begin{subequations}
\begin{align}
0 & \leftarrow J_Az_{k_n} - J_By_{k_n} \\
&= z_{k_n} - J_Az_{k_n} + R_Az_{k_n} - J_By_{k_n}\\
&= z_{k_n}-J_Az_{k_n} + y_{k_n}-J_By_{k_n}\\
&\weakly z+y-2x.
\end{align}
\end{subequations}
To summarize,
\begin{subequations}
\begin{align}
(z_{k_n},y_{k_n})\weakly (z,y)
\quad\text{and}\quad
(J_Az_{k_n},J_By_{k_n})\weakly (x,x),\\
(z_{k_n}-J_Az_{k_n}) + (y_{k_n}-J_By_{k_n})
\to -2x + z+ y = 0,\\
J_Az_{k_n}-J_By_{k_n}\to 0.
\end{align}
\end{subequations}
By Theorem~\ref{t:firm}, $J_Az=J_By=x$.
Hence $J_Az_{k_n} \weakly J_Az$. 
Since $x$ was an arbitrary weak cluster point of the bounded
sequence $(J_Az_n)_\nnn$, we conclude that $J_Az_n\weakly J_Az$.
\end{proof}

\begin{remark}
Generalizations of Fact~\ref{f:Svaiter} appear in
\cite{Svai10}, \cite{BC}, and a forthcoming preprint by 
Dr.~Patrick~L.~Combettes.
\end{remark}

\section*{Acknowledgment}
The author was partially supported by the Natural Sciences and
Engineering Research Council of Canada and
by the Canada Research Chair Program.

\end{document}